\documentclass[oneside,english]{amsart}
\usepackage[T1]{fontenc}
\usepackage[latin9]{inputenc}
\usepackage{color}
\usepackage{amsthm}
\usepackage{amssymb}

\makeatletter
\numberwithin{equation}{section}
\numberwithin{figure}{section}
\theoremstyle{plain}
\newtheorem{thm}{Theorem}[section]
  \theoremstyle{plain}
  \newtheorem{cor}[thm]{Corollary}
  \theoremstyle{definition}
  \newtheorem{defn}[thm]{Definition}
  \theoremstyle{plain}
  \newtheorem{fact}[thm]{Fact}
  \theoremstyle{remark}
  \newtheorem{rem}[thm]{Remark}
 \theoremstyle{definition}
  \newtheorem{example}[thm]{Example}
  \theoremstyle{remark}
  \newtheorem{claim}[thm]{Claim}
  \theoremstyle{plain}
  \newtheorem{lem}[thm]{Lemma}
  \theoremstyle{plain}
  \newtheorem{prop}[thm]{Proposition}

\usepackage{amsmath}

\usepackage{eulervm}
\usepackage{url}
\usepackage{enumerate}

\makeatother

\usepackage{babel}

\begin{document}
\begin{flushleft}
\global\long\def\NTPT{\operatorname{NTP}_{\operatorname{2}}}
\global\long\def\C{\mathfrak{C}}
\global\long\def\Aut{\operatorname{Aut}}
\global\long\def\tp{\operatorname{tp}}
\global\long\def\id{\operatorname{id}}
\global\long\def\ist{\operatorname{ist}}
\global\long\def\C{\mathfrak{C}}
\global\long\def\alt{\operatorname{alt}}
\global\long\def\st{\operatorname{st}}
\global\long\def\dom{\operatorname{dom}}
\global\long\def\acl{\operatorname{acl}}
\global\long\def\eq{\operatorname{eq}}

\par\end{flushleft}

\def\Ind#1#2{#1\setbox0=\hbox{$#1x$}\kern\wd0\hbox to 0pt{\hss$#1\mid$\hss} \lower.9\ht0\hbox to 0pt{\hss$#1\smile$\hss}\kern\wd0} 
\def\Notind#1#2{#1\setbox0=\hbox{$#1x$}\kern\wd0\hbox to 0pt{\mathchardef \nn="3236\hss$#1\nn$\kern1.4\wd0\hss}\hbox to 0pt{\hss$#1\mid$\hss}\lower.9\ht0 \hbox to 0pt{\hss$#1\smile$\hss}\kern\wd0} 

  \theoremstyle{definition}
  \newtheorem{defClaim}[thm]{Definition/Claim}

\begin{flushleft}
\global\long\def\ind{\mathop{\mathpalette\Ind{}}}
 \global\long\def\nind{\mathop{\mathpalette\Notind{}}}
\global\long\def\leftexp#1#2{{\vphantom{#2}}^{#1}{#2}}

\par\end{flushleft}

\title{Forking and dividing in $\NTPT$ theories}

\author{Artem Chernikov \and Itay Kaplan}\thanks{This work has been supported by the Marie Curie Early Stage Training Network MATHLOGAPS -- MEST-CT-2004-504029 -- and by the Marie Curie Research Training Network MODNET -- MRTN-CT-2004-512234}
\begin{abstract}
We prove that in theories without the tree property of the second
kind (which include dependent and simple theories) forking and dividing
over models are the same, and in fact over any extension base. As
an application we show that dependence is equivalent to bounded non-forking
assuming $\NTPT$.
\end{abstract}
\maketitle

\section{Introduction}

\subsection*{Background.\protect \\
}

The study of forking in the dependent (NIP) setting was initiated
by Shelah in full generality \cite{Sh783} and by Dolich in the case
of nice $o$-minimal theories \cite{Dol}. Further results appear
in \cite{Ad}, \cite{HP}, \cite{OnUs2} and \cite{Sta}. The main
trouble is that apparently non-forking independence outside of the
simple context no longer corresponds to a notion of dimension in any
possible way. Moreover it is neither symmetric nor transitive (at
least in the classical sense). However in dependent theories it corresponds
to invariance of types, which is undoubtedly a very important concept,
and it is a meaningful combinatorial tool.

\subsection*{Main results.\protect \\
}

The crucial property of forking in simple theories is that it equals
dividing (thus the useful concept -- forking -- becomes somewhat more
understandable in real-life situations). It is known that there are
dependent theories in which forking does not equal dividing in general
(for example in circular order over the empty set, see section \ref{sec:optimality-of-results}).
However there is a natural restatement of the question due to Anand
Pillay: whether forking and dividing are equal over models? After
failing to find a counter-example we decided to prove it instead.
And so the main theorem of the paper is:
\begin{thm}
\label{thm:MainThm}Let $T$ be an $\NTPT$ theory (a class which
includes dependent and simple theories). Then forking and dividing
over models are the same -- a formula $\varphi\left(x,a\right)$ forks
over a model $M$ iff it divides over it. 
\end{thm}
In fact, a more general result is attained. Namely that: 
\begin{thm}
\label{thm:MainThmGeneral}Let $T$ be $\NTPT$. Then for a set $A$,
the following are equivalent:
\begin{enumerate}
\item $A$ is an extension base for $\ind^{f}$ (non-forking) (see definition
\ref{def:basis}).
\item $\ind^{f}$ has left extension over $A$ (see definition \ref{def:PreInd}).
\item Forking equals dividing over $A$ (i.e. a formula $\varphi\left(x,b\right)$
divides over $A$ iff if forks over $A$).
\end{enumerate}
\end{thm}
So theorem \ref{thm:MainThm} is a corollary of \ref{thm:MainThmGeneral}
(types over models are finitely satisfiable, so (1) is true), and
of course:
\begin{cor}
If $T$ is $\NTPT$ and all sets are extension bases for non-forking,
then forking equals dividing. (This class contains simple theories,
$o$-minimal and $c$-minimal theories).
\end{cor}

\subsection*{The idea of the proof.\protect \\
}

The idea is to generalize the proof of the theorem in simple theories.
There, {}``Kim's lemma'' was the main tool. The lemma says, that
in a simple theory, if $\varphi\left(x,a\right)$ divides over $A$,
then \emph{every} Morley Sequence over $A$ (i.e. an indiscernible
sequence $\left\langle a_{i}\left|i<\omega\right.\right\rangle $
such that for all $i<\omega$, $\tp\left(a_{i}/Aa_{0}\ldots a_{i-1}\right)$
does not fork over $A$ and $a_{i}\equiv_{A}a$) witnesses this. As
there is no problem to construct Morley sequences over any set, one
shows that forking equals dividing by constructing a Morley sequence
that starts with the parameters of the formulas witnessing forking.\\
To prove the parallel result in the $\NTPT$ context, we find a
new notion of independence, $\ind^{\ist}$ such that every $\ind^{\ist}$-Morley
sequence witnesses dividing. Then we show that this notion satisfies
{}``existence over a model'', i.e. that for every $a$, $a\ind_{M}^{\ist}M$.
For this we shall need the so-called {}``broom lemma''. Essentially
it says that if a formula is covered by finitely many formulas arranged
in a \textquotedbl{}nice position\textquotedbl{}, then we can throw
away the dividing ones, by passing to an intersection of finitely
many conjugates.

\subsection*{Applications.\protect \\
}

We give some corollaries, among them that in dependent theories forking
is type definable, has left extension over models (answering a question
of Itai Ben Yaacov), and that if $p$ is a global $\varphi$ type
which is invariant over a model, then it can be extended to a global
type invariant over the same model (strengthening a result that appears
in \cite{HP}).\\
Hans Adler asked in \cite{Ad} whether NIP is equivalent to boundedness
of non-forking. In section \ref{sec:BddForking+NTP2} we show that
assuming $\NTPT$, this is indeed the case. This generalizes a well-known
analogous result describing the subclass of stable theories inside
the class of simple theories. Finally in section \ref{sec:optimality-of-results}
we present 2 examples that show that the $\NTPT$ assumption is needed,
and explain why we work over models. These are variants of an example
due to Martin Ziegler of a theory in which forking does not equal
dividing over models.

\subsection*{Further remarks.\protect \\
}

In \cite{ArtyomTemp}, we give an example of a theory with IP, such
that forking is bounded (moreover, a global type does not fork over
a set iff it is finitely satisfiable in this set). This, together
with the result appearing in section \ref{sec:BddForking+NTP2}, completely
solves Adler's question from \cite{Ad} mentioned above.

\subsection*{Acknowledgments.\protect \\
}

We would like to thank Alex Usvyatsov for many helpful discussions
and comments, Martin Ziegler for allowing us to include his example,
Itai Ben Yaacov, Frank Wagner, the entire Lyon logic group for a very
fruitful atmosphere, and the MATHLOGAPS/MODNET networks which made
our collaboration possible.\\
The first author thanks Berlin logic group for organizing the seminar
on NIP which triggered his interest for the questions concerned.\\
\textcolor{red}{}\\
Finally, both authors are grateful to Vissarion for leading their
thoughts in the right direction.

\section{\label{sec:Perliminaries}Preliminaries}

\subsection*{Notation.\protect \\
}

Notations are standard.\\
As usual, $T$ is a first order theory; $\C$ is the monster model
(a big saturated model); all sets are subsets of $\C$ of size smaller
than $\left|\C\right|$ and all models are elementary substructures
of $\C$.\\
We shall not always distinguish between sets and sequences, i.e.
$a$ can be a singleton, a set, an $n$-tuple or a sequence of any
length of members of $\C$.\\
The variables $x,y$ are singletons or finite sequences.\\
For sets $A,B$ we write $AB$ for the union, and for an element
(or a sequence) $a$, we write $Aa$ for $A\cup\left\{ a\right\} $
(or $A\cup\mbox{im}\left(a\right)$). In some contexts, $ab$ will
denote the concatenation of the sequences $a$ and $b$ (for instance
when we write $ab\equiv cd$). \\
For us, $I,J$ denote infinite sequences.\\
A global type is a type over $\C$.

\subsection*{Preliminaries on dependent theories.\protect \\
}

Let us recall:
\begin{defn}
A theory $T$ has the \emph{independence property} if there is a formula
$\phi(x,y)$ and tuples $\left\{ a_{i}\left|i<\omega\right.\right\} $,
$\left\{ b_{u}\left|u\subseteq\omega\right.\right\} $ (in $\C$)
such that $\phi(a_{i},b_{u})$ if and only if $i\in u$. $T$ is \emph{dependent}
iff it does not have the independence property (also known as \emph{NIP}).
\end{defn}

\begin{defn}
\label{def:altRank}The alternation rank of a formula: $\alt\left(\varphi\left(x,y\right)\right)=$
\[
=\max\left\{ n<\omega\left|\exists\left\langle a_{i}\left|i<\omega\right.\right\rangle \mbox{ indiscernible, }\exists b:\varphi\left(a_{i},b\right)\leftrightarrow\neg\varphi\left(a_{i+1},b\right)\mbox{ for }i<n-1\right.\right\} \]
\end{defn}
\begin{fact}
$T$ is dependent iff every formula has finite alternation rank.
\end{fact}
To the best of our knowledge, this fact first appeared in \cite{poiThUnstable},
and is an easy exercise in the definition.

\subsection*{Pre-independence relations, dividing and forking.\protect \\
}

To make the presentation clearer, we chose to follow the style of
Adler in \cite{adlerTh-2005}, and define an abstract notion of independence.
By a pre-independence relation we shall mean a ternary relation $\ind$
on sets which satisfies one or more of the properties below. For a
more general definition of a pre-independence relation see e.g. \cite[Section 5]{Ad}.
Note that since normally our relation is not symmetric many properties
can be formulated both on the left side and on the right side. 
\begin{defn}
\label{def:PreInd}A pre-independence relation $\ind$ is an invariant
ternary relation on sets. We write $a\ind_{A}b$ for: $a$ is $\ind$-independent
from $b$ over $A$. The following are the properties we consider
for a pre-independence relation:
\begin{enumerate}
\item Monotonicity: If $aa'\ind_{A}bb'$ then $a\ind_{A}b$.
\item Base monotonicity: If $a\ind_{A}bc$ then $a\ind_{Ab}c$. 
\item Transitivity on the left (over $A$): $a\ind_{Ab}c$ and $b\ind_{A}c$
implies $ab\ind_{A}c$. 
\item Right extension (over $A$): if $a\ind_{A}b$ then for all $c$ there
is $c'\equiv_{Ab}c$ such that $a\ind_{A}bc'$.
\item Left extension (over $A$): if $a\ind_{A}b$ then for all $c$ there
is $c'\equiv_{Aa}c$ such that $ac'\ind_{A}b$.
\end{enumerate}
\end{defn}
\begin{rem}
We shall not discuss independence relations, but for completeness
we mention that an independence relation is a pre-independence relation
that satisfies (1) -- (3) and symmetry (i.e. $a\ind_{A}b$ iff $b\ind_{A}a$).\end{rem}
\begin{defn}
We say that a pre-independence relation is \emph{standard} if it satisfies
(1) -- (4) from definition \ref{def:PreInd}.
\end{defn}

\begin{defn}
\label{def:basis}We say that $A$ is an extension base for a pre-independence
relation $\ind$ if for all $a$, $a\ind_{A}A$.
\end{defn}
Now let us recall the definition of forking and dividing.
\begin{defn}
(dividing) Let $A$ be be a set, and $a$ a tuple. We say that the
formula $\varphi\left(x,a\right)$ \emph{divide}s over $A$ iff there
is a number $k<\omega$ and tuples $\left\{ a_{i}\left|i<\omega\right.\right\} $
such that
\begin{enumerate}
\item $\tp\left(a_{i}/A\right)=\tp\left(a/A\right)$.
\item The set $\left\{ \varphi\left(x,a_{i}\right)\left|i<\omega\right.\right\} $
is $k$-inconsistent (i.e. every subset of size $k$ is not consistent).
\end{enumerate}
In this case, we say that a formula $k$-divides.\end{defn}
\begin{rem}
From Ramsey and compactness it follows that $\varphi\left(x,a\right)$
divides over $A$ iff there is an indiscernible sequence over $A$,
$\left\langle a_{i}\left|i<\omega\right.\right\rangle $ such that
$a_{0}=a$ and $\left\{ \varphi\left(x,a_{i}\right)\left|i<\omega\right.\right\} $
is inconsistent.\end{rem}
\begin{defn}
We say that a type $p$ divides over $A$ iff there is a finite conjunction
of formulas from $p$ which divides over $A$. The notation $a\ind_{A}^{d}b$
means $\tp\left(a/Ab\right)$ does not divide over $A$.\end{defn}
\begin{fact}
\label{fac:dividing} (see \cite[1.4]{SheSimple}) The following are
equivalent for every $T$:
\begin{enumerate}
\item $a\ind_{A}^{d}b$.
\item For every indiscernible sequence $I$ over $A$ such that $b\in I$,
there is an indiscernible sequence $I'$ such that $I'\equiv_{Ab}I$
and $I'$ is indiscernible over $Aa$.
\item For every indiscernible sequence $I$ over $A$ such that $b\in I$,
there is $a'$ such that $a'\equiv_{Ab}a$ and $I$ is indiscernible
over $Aa'$.
\end{enumerate}
\end{fact}
\begin{defn}
(forking) Let $A$ be be a set, and $a$ a tuple. 
\begin{enumerate}
\item Say that the formula $\varphi\left(x,a\right)$ \emph{forks} over
$A$ if there are formulas $\psi_{i}\left(x,a_{i}\right)$ for $i<n$
such that $\varphi\left(x,a\right)\vdash\bigvee_{i<n}\psi_{i}\left(x,a_{i}\right)$
and $\psi_{i}\left(x,a_{i}\right)$ divides over $A$ for every $i<n$.
\item Say that a type $p$ forks over $A$ if there is a finite conjunction
of formulas from $p$ which forks over $A$.
\item The notation $a\ind_{A}^{f}b$ means: $\tp\left(a/Ab\right)$ does
not fork over $A$.
\end{enumerate}
\end{defn}
Note that:
\begin{rem}
$ $
\begin{enumerate}
\item If $\varphi\left(x,a\right)$ divides over $A$ then it forks over
$A$.
\item If $M\supseteq A$ is an $\left|A\right|^{+}$ saturated model and
$p\in S\left(M\right)$ does not divide over $A$, then it does not
fork over $A$.
\end{enumerate}
\end{rem}

\begin{rem}
$\ind^{f}$ is standard (see, e.g. \cite[section 5]{Ad}).
\end{rem}
Two other pre-independence relations we shall use are $\ind^{u}$
(finite satisfiability -- the $u$ comes from {}``ultrafilter''),
and $\ind^{i}$ (invariance).
\begin{defn}
We write $a\ind_{A}^{u}b$ when $\tp\left(a/Ab\right)$ is finitely
satisfiable in $A$.\end{defn}
\begin{rem}
\label{cla:CoHeir}$\ind^{u}$ is standard and satisfies left extension
over models. Every model is an extension base for $\ind^{u}$.\end{rem}
\begin{proof}
The fact that $\ind^{u}$ is standard can be seen in e.g. \cite[section 5]{Ad}.
For left extension over models: Consider inheritance ($\ind^{h}$)
over a model $M$: $a\ind_{M}^{h}b$ iff $\tp\left(a/Mb\right)$ is
an heir over $M$, iff $b\ind_{M}^{u}a$. It is well known that $\ind^{h}$
satisfies right extension over models, so the result follows. The
fact that every model is an extension base follows from the fact that
filters can be extended to ultrafilters.
\end{proof}
Let us recall the definition of Lascar strong types.
\begin{defn}
$\Aut f_{L}\left(\C/A\right)$ is the subgroup of all automorphisms
of $\C$ generated by the set $\left\{ f\in\Aut\left(\C/M\right)\left|M\supseteq A\mbox{ is some small model}\right.\right\} $.
We write $a\equiv_{A}^{L}b$ ($a$ is Lascar equivalent to $b$, or
$a$ and $b$ have the same Lascar strong type) if there is $\sigma\in\Aut f_{L}\left(\C/A\right)$
taking $a$ to $b$.\end{defn}
\begin{fact}
\label{fac:Lascar}(See e.g. in \cite{AvivThesis}) The relation $\equiv_{A}^{L}$
is an equivalence relation, and in fact it is the finest invariant
equivalence relation with boundedly many classes. It is also defined
as the transitive closure of the relation $E\left(a,b\right)$ saying
that there is an indiscernible sequence over $A$ containing both
$a$ and $b$.
\end{fact}
Now we can define another pre-independence relation:
\begin{defn}
We say that $a\ind_{A}^{i}b$ iff there is is a global type $p$ extending
$\tp\left(a/Ab\right)$ which is Lascar invariant over $A$: for every
$c,d$ such that $c\equiv_{A}^{L}d$ and every formula $\varphi\left(x,y\right)$
over $A$, $\varphi\left(x,c\right)\in p$ iff $\varphi\left(x,d\right)\in p$.\end{defn}
\begin{rem}
\label{rem:PerservesIndisc} In general, by Fact \ref{fac:Lascar},
if $I$ is an indiscernible sequence over $A$ and $a\ind_{A}^{i}I$
then $I$ is indiscernible over $Aa$. So $a\ind_{A}^{i}b$ iff for
every finitely many indiscernible sequences over $A$, $I_{1},\ldots,I_{n}$,
there are sequences $I_{1}',\ldots,I_{n}'$ such that $\left\langle I_{1}'\ldots I_{n}'\right\rangle \equiv_{Ab}\left\langle I_{1}\ldots I_{n}\right\rangle $
and $I_{i}'$ is indiscernible over $Aa$. Hence, it is easy to see
that $\ind^{i}$ is standard. For more details, see \cite[Corollary 35]{Ad}.\\
In addition, over a model $M$, $\ind_{M}^{i}$ is non-splitting
(invariance) -- $a\ind_{M}^{i}b$ iff $\tp\left(a/Mb\right)$ can
be extended to a global invariant type over $M$.\end{rem}
\begin{defn}
We say that $\ind$ is at least as strong as $\ind'$ if for every
$a,b$ and $A$, $a\ind_{A}b\Rightarrow a\ind'_{A}b$.\end{defn}
\begin{example}
$\ind^{u}$ is at least as strong as $\ind^{i}$ which is at least
as strong as $\ind^{f}$. See claim below.
\end{example}
By the remark above, when $\ind$ is at least as strong as $\ind^{i}$,
if $I$ is indiscernible over $A$ and $a\ind_{A}I$ then $I$ is
indiscernible over $Aa$. In this case, we'll say that $\ind$ preserves
indiscernibility. In fact, these two are equivalent (i.e. to preserve
indiscernibility and to be as strong as $\ind^{i}$) for standard
pre-independence relations: it follows from right extension and the
criterion given in \ref{rem:PerservesIndisc}.
\begin{rem}
\label{rem:invUniqExt} If $N$ is $\left|A\right|^{+}$ saturated,
and $p\in S\left(N\right)$ is an $A$-invariant type, then $p$ has
a unique extension to a global $A$-invariant type.\end{rem}
\begin{claim}
\label{cla:indI=00003DindF} $\ind^{i}$ is at least as strong as
$\ind^{f}$. If $T$ is dependent, then $\ind^{i}=\ind^{f}$.\end{claim}
\begin{proof}
The first statement is clear, and the second appears in \cite{Sh783}
and also in \cite{Ad}. 
\end{proof}

\subsection*{Generating indiscernible sequences.\protect \\
}

Recall the following fact:
\begin{fact}
\label{fac:morlySeq} Assume that $p$ is global $A$-invariant type.
Then $p$ generates an indiscernible sequence over $A$: $a_{0}\models p|_{A}$,
$a_{i+1}\models p|_{Aa_{0}\ldots a_{i}}$. The type of this indiscernible
sequence depends only on $p$, and will be denoted by $p^{\left(\omega\right)}|_{A}\in S^{\left(\omega\right)}\left(A\right)$.
The type we get after $n$ steps is denoted by $p^{\left(n\right)}|_{A}\in S^{n}\left(A\right)$. \end{fact}
\begin{defn}
$ $
\begin{enumerate}
\item A type $p$ is $\ind$-free over $A$ if for any $b$ such that $Ab\subseteq\dom\left(p\right)$
and every $a\models p|_{Ab}$, $a\ind_{A}b$.
\item A Morley sequence $\left\langle a_{i}\left|i<\omega\right.\right\rangle $
for $\ind$ with base $A$ over $B\supseteq A$ is an indiscernible
sequence over $B$, such that for all $i$, $a_{i}\ind_{A}Ba_{0}\ldots a_{i-1}$.
\end{enumerate}
\end{defn}
Note that if a global type $p$ is $\ind$-free and invariant over
$A$, then for every $B\supseteq A$, the sequence $p$ generates
over $B$ is a Morley sequence with base $A$ over $B$.

\subsection*{$\NTPT$ Theories.}
\begin{defn}
\label{def:TP2} A theory $T$ has $\mbox{TP}_{\mbox{2}}$ (the tree
property of the second kind) if there exists a formula $\varphi(x,y)$,
a number $k<\omega$ and an array of elements $\left\langle a_{i}^{j}\left|i,j<\omega\right.\right\rangle $
(in $\C$) such that:
\begin{itemize}
\item Every row is $k$-inconsistent: for every $i<\omega$ and $j_{0},\ldots,j_{k-1}<\omega$,
$\C\models\neg\left(\exists x\bigwedge_{l<k}\varphi\left(x,a_{i}^{j_{l}}\right)\right)$.
\item Every vertical path is consistent: for every function $\eta:\omega\to\omega$,
the set $\left\{ \varphi\left(x,a_{i,\eta\left(i\right)}\right)\left|i<\omega\right.\right\} $
is consistent.
\end{itemize}
We say that $T$ is $\NTPT$ when it does not have $\mbox{TP}_{\mbox{2}}$.\end{defn}
\begin{fact}
Every dependent theory as well as every simple one is $\NTPT$.\end{fact}
\begin{proof}
The tree property of the second kind implies the tree property (so
every simple theory is $\NTPT$) and the Independence property.
\end{proof}
The tree property of the second kind was defined in \cite{SheSimple}.
There it is proved that a theory is non-simple (has the tree property)
iff it has the tree property of the first kind (which we shall not
define here) or the the tree property of the second kind.

\section{\label{sec:Main-results}Main results}

\subsection{The Broom lemma.\protect \\
}

We start with the main technical lemma. Here there are no assumptions
on $T$.
\begin{lem}
\label{lem:Broom}Suppose that $\ind$ satisfies all properties from
\ref{def:PreInd} but we demand that it satisfies left extension only
over $A$, and in addition that it preserves indiscernibility. Assume
that \[
\alpha\left(x,e\right)\vdash\psi\left(x,c\right)\vee\bigvee_{i<n}\varphi_{i}\left(x,a_{i}\right)\]
where

\begin{enumerate}[(A)]\item\label{enu:kdiv}For $i<n$, the formula
$\varphi_{i}\left(x,a_{i}\right)$ $k$-divides over $A$, as witnessed
by the indiscernible sequence $I_{i}=\left\langle a_{i,l}\left|l<\omega\right.\right\rangle $
where $a_{i,0}=a_{i}$.

\item\label{enu:ind}For each $i<n$ and $1\leq l$, $a_{i,l}\ind_{A}a_{i,<l}I_{<i}$
where $a_{i,<l}=a_{i,0}\ldots a_{i,l-1}$, and $I_{<i}=I_{0}\ldots I_{i-1}$.

\item$c\ind_{A}I_{<n}$.

\end{enumerate}

\underbar{Then} for some $m<\omega$ there is $\left\{ e_{i}\left|i<m\right.\right\} $
with $e_{i}\equiv_{A}e$ for $i<m$ and $\bigwedge_{i<m}\alpha\left(x,e_{i}\right)\vdash\psi\left(x,c\right)$.
In particular, if $\psi\left(x,c\right)=\bot$ (i.e. $\forall x\left(x\neq x\right)$)
, then $\left\{ \alpha\left(x,e_{i}\right)\left|i<m\right.\right\} $
is inconsistent. 

\end{lem}
\begin{proof}
By induction on $n$. For $n=0$ there is nothing to prove. \\
Assume that the claim is true for $n$ and we prove it for $n+1$.
Let $b_{0}=a_{n,0}\ldots a_{n,k-2}$ and $b_{1}=a_{n,1}\ldots a_{n,k-1}$
(where $k$ is from (\ref{enu:kdiv})). Since $\ind$ preserves indiscernibility,
as $c\ind_{A}I_{n}$ we have \[
cb_{1}\equiv_{A}cb_{0}.\]
 We build by induction on $j<k$ sequences $\left\langle I_{<n}^{l,j}\left|l\leq j\right.\right\rangle $
(so $I_{<n}^{l,j}=I_{0}^{l,j}\ldots I_{n-1}^{l,j}$) such that:
\begin{enumerate}
\item $I_{<n}^{l,j}=I_{0}^{l,j}\ldots I_{n-1}^{l,j}$ and each $I_{i}^{l,j}$
is of the same length as $I_{i}$,
\item $I_{<n}^{0,j}=I_{<n}$.
\item $I_{<n}^{l,j}ca_{n,l}\equiv_{A}I_{<n}^{0,j}ca_{n,0}$ for all $l\leq j$
and
\item For all $0\leq l<j$, $cI_{<n}^{j,j}I_{<n}^{j-1,j}\ldots I_{<n}^{l+1,j}\ind_{A}I_{<n}^{l,j}$
and $c\ind_{A}I_{<n}^{j,j}$ (which already follows from the previous
clauses).
\end{enumerate}
For $j=0$, use (2): $I_{<n}^{0,0}=I_{<n}$.\\
So suppose we have this sequence for $j$ and we build it for $j+1<k$.
\\
By (2), let $I_{<n}^{0,j+1}=I_{<n}$.\\
As $cb_{1}\equiv_{A}cb_{0}$ we can find some $J_{<n}^{l,j+1}$
for $1\leq l\leq j+1$ so that:\begin{equation}
J_{<n}^{j+1,j+1}J_{<n}^{j,j+1}\ldots J_{<n}^{1,j+1}cb_{1}\equiv_{A}I_{<n}^{j,j}I_{<n}^{j-1,j}\ldots I_{<n}^{0,j}cb_{0}.\tag{I}\label{eq:Aut}\end{equation}
By transitivity on the left and base monotonicity (and by (\ref{enu:ind}))
we have $cb_{1}\ind_{A}a_{n,0}I_{<n}$, and by left extension we can
find $\left\langle I_{<n}^{l,j+1}\left|1\leq l\leq j+1\right.\right\rangle $
such that \begin{equation}
I_{<n}^{j+1,j+1}I_{<n}^{j,j+1}\ldots I_{<n}^{1,j+1}cb_{1}\equiv_{A}J_{<n}^{j+1,j+1}J_{<n}^{j,j+1}\ldots J_{<n}^{1,j+1}cb_{1}\tag{II}\label{eq:LeftExt}\end{equation}
and\begin{equation}
\left\langle I_{<n}^{l,j+1}\left|1\leq l\leq j+1\right.\right\rangle cb_{1}\ind_{A}a_{n,0}I_{<n}.\tag{III}\label{eq:indi}\end{equation}

And so we have constructed $\left\langle I_{<n}^{l,j+1}\left|l\leq j+1\right.\right\rangle $.\\
Note that from equations (\ref{eq:Aut}) and (\ref{eq:LeftExt})
it follows that\begin{equation}
I_{<n}^{j+1,j+1}I_{<n}^{j,j+1}\ldots I_{<n}^{1,j+1}cb_{1}\equiv_{A}I_{<n}^{j,j}I_{<n}^{j-1,j}\ldots I_{<n}^{0,j}cb_{0}.\tag{IV}\label{eq:together}\end{equation}
Now to check that we have our conditions satisfied:\\
 (1) and (2) follows directly from construction.\\
(3): First of all, $I_{<n}ca_{n,0}\equiv_{A}I_{<n}^{1,j+1}ca_{n,1}$
by equation (\ref{eq:together}). For $1\leq l\leq j$, \[
I_{<n}ca_{n,0}\equiv_{A}I_{<n}^{l,j}ca_{n,l}\]
 by the hypothesis regarding $j$. By (\ref{eq:together}), \[
I_{<n}^{l,j}ca_{n,l}\equiv_{A}I_{<n}^{l+1,j+1}ca_{n,l+1}\]
 and so we have (3) for $l\leq j+1$. \\
(4) follows from (\ref{eq:indi}), the invariance of $\ind$ and
induction. \\
So, for $j=k-1$ we have $\left\langle I_{<n}^{l,k-1}\left|l\leq k-1\right.\right\rangle $.
We shall now use only this last sequence.\\
 There are some $\left\langle e_{l}\left|l<k\right.\right\rangle $
such that $e_{0}=e$ and for $0<l<k$, $e_{l}I_{<n}^{l,k-1}ca_{n,l}\equiv_{A}eI_{<n}ca_{n,0}$,
so applying some automorphism fixing $Ac$, we replace $a_{n,0}$
by $a_{n,l}$, $e$ by $e_{l}$ and $I_{<n}$ by $I_{<n}^{l,k-1}$.
So we get \[
\alpha\left(x,e_{l}\right)\vdash\psi\left(x,c\right)\vee\bigvee_{i<n}\varphi_{i}\left(x,a_{i}^{l,k-1}\right)\vee\varphi_{n}\left(x,a_{n,l}\right)\]
 where $a_{i}^{l,k-1}$ starts $I_{i}^{l,k-1}$. Hence $\alpha^{0}=\bigwedge_{l<k}\alpha\left(x,e_{l}\right)$
implies the conjunction of these formulas. But as $I_{n}$ witnesses
that $\varphi_{n}\left(x,a_{n}\right)$ is $k$ dividing, we have
the following:\[
\alpha^{0}\vdash\psi\left(x,c\right)\vee\bigvee_{i<n,l<k}\varphi_{i}\left(x,a_{i}^{l,k-1}\right).\]
 Define a new formulas $\psi^{r}\left(x,c^{r}\right)=\psi\left(x,c\right)\vee\bigvee_{i<n,r\leq l<k}\varphi_{i}\left(x,a_{i}^{l,k-1}\right)$
for $r\leq k$. By induction on $r\leq k$, we find $\alpha^{r}$
such that $\alpha^{r}$ is a conjunction of conjugates over $A$ of
$\alpha\left(x,e\right)$, and $\alpha^{r}\vdash\psi^{r}\left(x,c^{r}\right)$.
It will follow of course, that $\alpha^{k}\vdash\psi\left(x,c\right)$
as desired. For $r=0$, we already found $\alpha^{0}$. Assume we
found $\alpha^{r}$, so we have\\
 \[
\alpha^{r}\vdash\psi^{r+1}\left(x,c^{r+1}\right)\vee\bigvee_{i<n}\varphi_{i}\left(x,a_{i}^{r,k-1}\right)\]
One can easily see that the hypothesis of the lemma is true for this
implication (where $c=c^{r+1}$, and $I_{i}=I_{i}^{r,k-1}$) so by
the induction hypothesis (on $n$), there is some $\alpha^{r+1}$(which
is a conjunction of conjugates of $\alpha^{r}$ over $A$, and so
also of $\alpha$) such that $\alpha^{r+1}\vdash\psi^{r+1}\left(x,c^{r+1}\right)$. \end{proof}
\begin{defn}
\label{def:quasiDiv} We say that a formula $\alpha\left(x,e\right)$
quasi-divides over $A$ if there are $m<\omega$ and $\left\{ e_{i}\left|i<m\right.\right\} $
such that $e_{i}\equiv_{A}e$ and $\left\{ \alpha\left(x,e_{i}\right)\left|i<m\right.\right\} $
is inconsistent.
\end{defn}
So this lemma shows that under certain conditions, a forking formula
also quasi-divides.
\begin{rem}
The name of this lemma is due to its method of proof, which reminded
the authors (and also Itai Ben Yaacov who thought of the name) of
a sweeping operation.
\end{rem}

\subsection{On pre-independence relations in $\NTPT$.\protect \\
}

\subsection*{Existence of global free co-free types.\protect \\
}

The title of this section may seem a bit mysterious, but it will become
clearer with the next Proposition. Let $T$ be any theory. 
\begin{defn}
Let $\ind$ be a pre-independence relation. We say that $\ind$ has
\emph{finite character} if whenever $a\nind_{B}b$, there is a formula
$\varphi\left(x\right)$ over $Bb$ such that $\varphi\left(a\right)$
and for all $a'$ if $\varphi\left(a'\right)$ then $a'\nind_{B}b$.\end{defn}
\begin{rem}
This definition is taken from \cite{Ad}, where it is called strong
finite character, but since there is no room for confusion, we decided
to omit {}``strong''.\end{rem}
\begin{example}
All the pre-independence relations we mentioned satisfy this: $\ind^{f}$,
$\ind^{u}$ and $\ind^{i}$.\end{example}
\begin{prop}
\label{cla:symmetrizers}Assume that $\ind$ is a standard pre-independence
relation with finite character. Assume that $B$ is an extension base
for $\ind$ and that if $\varphi\left(x,a\right)$ forks over $B$,
then $\varphi\left(x,a\right)$ quasi-divides over $B$ (see \ref{def:quasiDiv};
in this case we say that forking implies quasi dividing over $B$).\\
\underbar{Then}: for every type $p$ over $B$,
\begin{enumerate}
\item There exists a global extension $q$, $\ind$-free over $B$, such
that for every $C\supseteq B$ and every $c\models q|_{C}$, $C\ind_{B}^{f}c$.
\item There exists a global extension $q'$ that doesn't fork over $B$
(i.e. $\ind^{f}$-free over $B$), such that for every $C\supseteq B$
and every $c\models q'|_{C}$, $C\ind_{B}c$.
\end{enumerate}
\end{prop}
\begin{proof}
(1): Let $a\models p$. By finite character, it is enough to see that
the following set is consistent\begin{eqnarray*}
p\left(x\right) & \cup & \left\{ \neg\varphi\left(x,b\right)\left|\varphi\left(x,y\right)\mbox{ is over }B\,\&\, b\in\C\,\&\,\varphi\left(a,y\right)\mbox{ forks over }B\right.\right\} \\
 & \cup & \left\{ \neg\psi\left(x,d\right)\left|\psi\left(x,z\right)\mbox{ is over }B\,\&\, d\in\C\,\&\forall c\left[\psi\left(c,d\right)\Rightarrow c\nind_{B}d\right]\right.\right\} .\end{eqnarray*}
Since then every global type $q$ that contains this set will suffice.\\
Indeed: assume not, then we have an implication of the form\[
p\vdash\bigvee_{i<n}\varphi_{i}\left(x,b_{i}\right)\vee\bigvee_{j<m}\psi_{j}\left(x,d_{j}\right)\]
 where $\varphi_{i}\left(x,y_{i}\right)$, $\psi_{j}\left(x,z_{j}\right)$
formulas over $B$, $\forall c\left[\psi_{j}\left(c,d_{j}\right)\Rightarrow c\nind_{B}d_{j}\right]$
and $\varphi_{i}\left(a,y_{i}\right)$ forks over $B$.\\
Note that $\bigvee_{i<n}\varphi_{i}\left(a,y_{i}\right)$ forks
over $B$, so we may assume $n=1$.\\
By assumption, $\varphi_{0}\left(a,y\right)$ quasi-divides over
$B$, so there are $h_{0},\ldots,h_{k-1}$ such that $h_{i}\equiv_{B}a$
and $\left\{ \varphi_{0}\left(h_{i},y\right)\left|i<k\right.\right\} $
is inconsistent. Denote $h=h_{0}h_{1}\ldots h_{k-1}$ and $r\left(x_{0},\ldots,x_{k-1}\right)=\tp\left(h/B\right)$.
Then \[
r\upharpoonright x_{i}\vdash\varphi_{0}\left(x_{i},b\right)\vee\bigvee_{j<m}\psi_{j}\left(x_{i},d_{j}\right).\]
So \[
r\vdash\bigwedge_{i<k}\left[\varphi_{0}\left(x_{i},b\right)\vee\bigvee_{j<m}\psi_{j}\left(x_{i},d_{j}\right)\right].\]
But\[
r\vdash\neg\exists z\left(\bigwedge_{i<k}\varphi_{0}\left(x_{i},z\right)\right),\]
so $r\vdash\bigvee_{i<k,j<m}\psi_{j}\left(x_{i},d_{j}\right)$.\\
The set $B$ is an extension base for $\ind$, so $h\ind_{B}B$,
and by right extension there is $h'\equiv_{B}h$ such that $h'\ind_{B}\left\{ d_{j}\left|j<m\right.\right\} $.
It follows that there are $i,j$ such that $\psi_{j}\left(h_{i}',d_{j}\right)$.
This is a contradiction to the choice of $\psi_{j}$.

(2): The proof is very similar. Let $a\models p$. We must show that\begin{eqnarray*}
p\left(x\right) & \cup & \left\{ \neg\varphi\left(x,b\right)\left|\varphi\left(x,y\right)\mbox{ is over }B\,\&\, b\in\C\,\&\,\varphi\left(x,b\right)\mbox{ forks over }B\right.\right\} \\
 & \cup & \left\{ \neg\psi\left(x,d\right)\left|\psi\left(x,z\right)\mbox{ is over }B\,\&\, d\in\C\,\&\forall c\left[\psi\left(a,c\right)\Rightarrow c\nind_{B}a\right]\right.\right\} \end{eqnarray*}
is consistent. If not, then $p\vdash\bigvee_{i<n}\varphi_{i}\left(x,b_{i}\right)\vee\bigvee_{j<m}\psi_{j}\left(x,d_{j}\right)$
and we may assume $n=1$. As $\varphi_{0}\left(x,b_{0}\right)$ forks
over $B$, it quasi-divides over $B$, so there are $e_{0},\ldots,e_{k-1}$
such that $e_{i}\equiv_{B}b_{0}$ and $\left\{ \varphi\left(x,e_{i}\right)\left|i<k\right.\right\} $
is inconsistent. Let $\bar{d}=\left\langle d_{i,j}\left|j<m\right.\right\rangle $
be such that $\bar{d}_{i}e_{i}\equiv_{B}\bar{d}b_{0}$. As $p$ is
over $B$, for every $i<k$,\[
p\vdash\varphi_{0}\left(x,e_{i}\right)\vee\bigvee_{j<m}\psi_{j}\left(x,d_{i,j}\right).\]
So it follows that $p\vdash\bigvee_{i,j}\psi_{j}\left(x,d_{i,j}\right)$.
Denote $\bar{d}'=\left\langle d_{i,j}\left|i<k,j<m\right.\right\rangle $.
As $B$ is an extension base for $\ind$, $\bar{d}'\ind_{B}B$, and
by right extension, wlog $\bar{d}'\ind_{B}a$. So there are $i,j$
such that $\psi_{j}\left(a,d'_{i,j}\right)$ which contradicts the
choice of $\psi_{j}$. 
\end{proof}
The following pre-independence relation is instrumental in the proof
of the main theorem.
\begin{defn}
\label{def:strictInv}We say that $\tp\left(a/Bb\right)$ is strictly
invariant over $B$ (denoted by $a\ind_{B}^{\ist}b$) if there is
a global extension $p$, which is Lascar invariant over $B$ (so $a\ind_{B}^{i}b$)
and for any $C\supseteq Bb$, if $c\models p|_{C}$ then $C\ind_{B}^{f}c$. \end{defn}
\begin{rem}
\label{rem:StrictNF}$ $\end{rem}
\begin{enumerate}
\item $\ind^{\ist}$ satisfies extension, invariance and monotonicity.
\item Strictly invariant types are a special case of strictly non-forking
types. We say that $\tp\left(a/Bb\right)$ strictly does not fork
over $B$ (denoted by $a\ind_{B}^{\st}b$) if there is a global extension
$p$, which does not fork over $B$, and for any $C\supseteq B$,
if $c\models p|_{C}$ then $C\ind_{B}^{f}c$. They coincide in dependent
theories, and in stable theories they are the same as non-forking.
The notion originated in \cite[5.6]{Sh783}. More on strict non-forking
can be found in \cite{Us1} and in \cite{AlexTemp}. 
\end{enumerate}
As $\ind^{i}$ has finite character, we conclude from (1) in Proposition
\ref{cla:symmetrizers} that:
\begin{cor}
\label{cor:istExist}Assume forking implies quasi dividing over $B$
and that $B$ is an extension base for $\ind^{i}$. Then $B$ is an
extension base for $\ind^{\ist}$.
\end{cor}

\subsection*{Working with an abstract pre-independence relation.\protect \\
}

Here we shall prove the following theorem:
\begin{thm}
\label{thm:MainAbs}Let $T$ be $\NTPT$. Then (1) implies (2) where:
\begin{enumerate}
\item There exists a standard pre-independence relation $\ind$ with left
extension over $B$, which preserves indiscernibility over $B$ and
such that $B$ is an extension base for it.
\item Forking equals dividing over $B$. 
\end{enumerate}
In addition, if $T$ is dependent then (1) and (2) are equivalent. 
\end{thm}

\subsection*{(1) implies (2).\protect \\
}

So assume $T$ is $\NTPT$, and that $\ind$ is a pre-independence
relation as in (1). We do not need left extension for this next claim:
\begin{lem}
\label{cla:globalTypeWitness}Assume $\varphi\left(x,a\right)$ divides
over $B$. Then there is a model $M\supseteq B$ and a global $\ind$-free
type over $B$, $p\in S\left(\C\right)$, extending $\tp\left(a/M\right)$,
such that every Morley sequence generated by $p$ over $M$ (as in
\ref{fac:morlySeq}) witnesses that $\varphi\left(x,a\right)$ divides.\end{lem}
\begin{proof}
Let $I=\left\langle b_{i}\left|i<\omega\right.\right\rangle $ be
a $B$-indiscernible sequence that witnesses $k$ dividing of $\varphi\left(x,a\right)$.
Let $N$ be a $\left(\left|B\right|+\left|T\right|\right)^{+}$ saturated
model containing $B$. By compactness we may assume that the length
of $I$ is $\left(2^{\left|N\right|+\left|T\right|}\right)^{+}$.
As $B$ is an extension base, we may assume that $I\ind_{B}N$. The
number of types over $N$ is bounded by $2^{\left|N\right|+\left|T\right|}$,
so $I$ has infinitely many elements with the same type $p$ over
$N$, and wlog they are the first $\omega$. Replace $I$ with $I\upharpoonright\omega$.
Let $B\subseteq M\subseteq N$ be any model such that $\left|M\right|\leq\left|B\right|+\left|T\right|$.
\\
Let $Q\left(x_{0},x_{1},\ldots\right)=\tp\left(I/N\right)$. Then
$Q$ is an invariant type over $M$ (as $M$ is a model and $Q$ is
Lascar invariant over $B$), and so is $p\left(x_{i}\right)=Q\upharpoonright x_{i}$.
By saturation, we can define a sequence $\left\langle I_{i}\left|i<\omega\right.\right\rangle $
in $N$ as in \ref{fac:morlySeq}: $I_{0}\models Q|_{M}$, $I_{i+1}\models Q|_{MI_{0}\ldots I_{i}}$.
Then $\left\langle I_{i}\left|i<\omega\right.\right\rangle $ is an
indiscernible sequence. Let $I_{i}=\left\langle a_{i,j}\left|j<\omega\right.\right\rangle $.
It follows that for every $\eta:\omega\to\omega$, $a_{0,\eta\left(0\right)}a_{1,\eta\left(1\right)}\ldots\equiv_{M}a_{0,0}a_{1,0}\ldots$,
as both sequences satisfy the type $p^{\left(\omega\right)}|_{M}$.
\\
As $T$ is $\NTPT$, $\left\{ \varphi\left(x,a_{i,0}\right)\left|i<\omega\right.\right\} $
is inconsistent (otherwise $\left\{ \varphi\left(x,a_{i,j}\right)\left|i,j<\omega\right.\right\} $
witnesses that $T$ has the tree property of the second kind because
of the choice of $I$). \\
By \ref{rem:invUniqExt}, the type $p$ has a unique extension
to a global $\ind$-free type over $B$ (which we shall also call
$p$).\\
Let $a'\models p|_{M}$, then $a'\equiv_{B}a$, so after applying
an automorphism over $B$ (and changing $M$), we may assume that
$p$ extends $\tp\left(a/M\right),$ and it is the required type:
it is $\ind$-free (as $Q$ is), and there is a Morley sequence generated
by $p$ that witnesses dividing, so every such sequence does so as
well.\end{proof}
\begin{cor}
\label{cor:NTP2ForkImWeakDiv}Forking implies quasi dividing over
$B$.\end{cor}
\begin{proof}
Suppose $\varphi\left(x,a\right)$ forks over $B$, then $\varphi\left(x,a\right)\vdash\bigvee_{i<n}\varphi_{i}\left(x,a_{i}\right)$
where for all $i<n$, $\varphi_{i}\left(x,a_{i}\right)$ divides over
$B$. By Lemma \ref{cla:globalTypeWitness}, for $i<n$, there are
models $M_{i}\supseteq B$ and types $p_{i}$ which are global $\ind$-free
extension of $\tp\left(a_{i}/B\right)$. Let $I_{0}$ be some indiscernible
sequence witnessing dividing of $\varphi_{0}\left(x,a_{0}\right)$.
For $0<i$, let $I_{i}=\left\langle a_{i,l}\left|l<\omega\right.\right\rangle $
be a Morley sequence generated by $p_{i}$ as follows: $a_{i,0}=a_{i}\models p_{i}|_{M_{i}}$,
and for all $j>0$, $a_{i,l+1}\models p_{i}|_{M_{i}I_{<i}a_{i,\leq l}}$.
This will set us in the situation of the broom lemma \ref{lem:Broom}
hence $\varphi$ quasi-divides over $B$.
\end{proof}
For the next claims, let $A$ be any set.\\
The importance of $\ind^{\ist}$ lies in the following lemma, which
is analogous to {}``Kim's Lemma'' (see \cite[2.1]{KimForking}).
\begin{lem}
\label{lem:Kim}If $\varphi\left(x,a\right)$ divides over $A$, and
$\left\langle b_{i}\left|i<\omega\right.\right\rangle $ is a sequence
satisfying $b_{i}\equiv_{A}a$ and $b_{i}\ind_{A}^{\ist}b_{<i}$.
Then $\left\{ \varphi\left(x,a_{i}\right)\left|i<\omega\right.\right\} $
is inconsistent. In particular, if $\left\langle b_{i}\left|i<\omega\right.\right\rangle $
is an indiscernible sequence then it witnesses dividing of $\varphi\left(x,a\right)$. \end{lem}
\begin{proof}
Wlog $b_{0}=a$. Let $I$ be an indiscernible sequence witnessing
the dividing of $\varphi\left(x,a\right)$ over $A$. We build by
induction on $n$ sequences $I_{i}=\left\langle a_{i,j}\left|j<\omega\right.\right\rangle $
for $i<n$ such that
\begin{itemize}
\item Each $I_{i}$ is indiscernible over $AI_{<i}a_{>i,0}$ (where $a_{>i,0}=a_{i+1,0}\ldots a_{n-1,0}$).
\item For $i<\omega$, $I_{i}\equiv_{A}I$.
\item $a_{i,0}=b_{i}$.
\end{itemize}
This is enough, because then by compactness we can find an infinite
such array and then if $\left\{ \varphi\left(x,b_{i}\right)\left|i<\omega\right.\right\} $
is consistent, we reach a contradiction to $\NTPT$: In the infinite
array $\left\langle a_{i,j}\left|i,j<\omega\right.\right\rangle $,
for every function $\eta:\omega\to\omega$ and every $n$, one may
show by decreasing induction on $i\leq n$ (starting with $i=n$),
that\[
a_{0,\eta\left(0\right)}\ldots a_{n-1,\eta\left(n-1\right)}\equiv_{A}a_{0,\eta\left(0\right)}\ldots a_{i-1,\eta\left(i-1\right)}a_{i,0}\ldots a_{n-1,0}.\]
And this shows that every vertical path has the same type, but each
row is $k$-inconsistent for the same $k$ (because $I_{i}\equiv_{A}I$).
\\
For $n\leq1$ it is clear. Suppose we have built these sequences
up to $n$ and we consider $n+1$. Denote our array of $n$ rows by
$I_{<n}$. By right extension, there is $J_{<n}\equiv_{Ab_{<n}}I_{<n}$
such that $b_{n}\ind_{A}^{\ist}J_{<n}$. Hence also $J_{<n}\ind_{A}^{f}b_{n}$.
As $b_{n}\equiv_{A}a$, there is an indiscernible sequence $I'\equiv_{A}I$
starting with $b_{n}$. By \ref{fac:dividing}, there is an $A$-indiscernible
sequence $J_{n}$ such that $J_{n}\equiv_{Ab_{n}}I'$ and $J_{n}$
is indiscernible over $J_{<n}$. \\
Now it is easy to check that the conditions we demanded are met
with this new array. The only non-trivial one is the first condition:
$J_{n}$ is indiscernible over $J_{<n}$ by construction. For every
$i<n$, $J_{i}$ is indiscernible over $AJ_{<i}b_{>i}$ by the induction
hypothesis (where $b_{>i}=b_{i+1}\ldots b_{n-1}$). As $b_{n}\ind_{A}^{i}J_{<n}$,
by the base monotonicity of $\ind^{i}$ it follows that $b_{n}\ind_{AJ_{<i}b_{>i}}^{i}J_{i}$,
and as $\ind^{i}$ preserves indiscernibility, it follows that $J_{i}$
is indiscernible over $AJ_{<i}b_{>i}b_{n}$. \end{proof}
\begin{rem}
In fact we need less than Lemma \ref{lem:Kim}. For our needs, it
suffices to see that if $\varphi\left(x,a\right)$ divides over $A$,
and there exists $p$, a global $\ind$-free type over $A$, containing
$\tp\left(a/A\right)$, then every Morley sequence $p$ generates
(over a model $M\supseteq A$) witnesses dividing. The proof of this
fact is a bit easier: Assume that $I$ witnesses dividing, and that
$N$ is $\left|M\right|^{+}$ saturated. Let $c\models p|_{N}$. Then
$c\ind_{A}^{\ist}N$ and in particular $N\ind_{A}^{f}c$, so (by \ref{fac:dividing})
we may find $I'$ such that $cI'\equiv_{A}aI$ and $I'$ is indiscernible
over $N$. Now, as in the proof of \ref{cla:globalTypeWitness}, we
define $I_{i}\models\tp\left(I'/N\right)|_{MI_{<i}}$ in $N$. Then,
every vertical path realizes the type $p^{\left(\omega\right)}|_{M}$
and we get a contradiction. \end{rem}
\begin{cor}
If $A$ is an extension base for $\ind^{ist}$ , then forking equals
dividing over $A$.\end{cor}
\begin{proof}
Suppose $\varphi\left(x,a\right)\vdash\bigvee_{i<n}\varphi_{i}\left(x,a_{i}\right)$,
each $\varphi_{i}\left(x,a_{i}\right)$ divides over $A$. Let $\bar{a}=aa_{0}\ldots a_{n-1}$
and let $p=\tp\left(\bar{a}/A\right)$. As $\bar{a}\ind_{A}^{\ist}A$,
by definition there is $q$, a global $\ind^{\ist}$-free type over
$A$, containing $p$.\\
Let $\left\langle \bar{a}^{j}=a^{j}a_{0}^{j}\ldots a_{n-1}^{j}\left|j<\omega\right.\right\rangle $
be a Morley sequence generated by $q$ over a model $M$ containing
$A$. It is enough to see that $\left\{ \varphi\left(x,a^{j}\right)\left|j<\omega\right.\right\} $
is inconsistent (as it is an indiscernible sequence whose elements
have the same type as $a$ over $A$). \\
If this set is consistent, let $c$ realize it. Then for all $j<\omega$,
there is $i_{j}<n$ such that $\varphi\left(c,a_{i_{j}}^{j}\right)$,
so there is $\iota<n$ and infinitely many $j$'s such that $\iota=i_{j}$.
Then $\left\{ \varphi_{i_{0}}\left(x,a_{\iota}^{j}\right)\left|i_{j}=\iota\right.\right\} $
is consistent -- a contradiction to \ref{lem:Kim}.\end{proof}
\begin{lem}
\label{lem:istExist}The set $B$ (from our assumptions) is an extension
base for $\ind^{\ist}$ .\end{lem}
\begin{proof}
Forking implies quasi-dividing over $B$ by \ref{cor:NTP2ForkImWeakDiv},
and $B$ is an extension base for $\ind^{i}$ by our assumption (because
$\ind$ is at least as strong as $\ind^{i}$), so the lemma follows
immediately from \ref{cor:istExist}.
\end{proof}
Summing up, we have
\begin{cor}
\label{cor:FinalCor}Forking equals dividing over $B$.
\end{cor}
By this we have proved one direction of Theorem \ref{thm:MainAbs}.

\subsection*{(2) implies (1).\protect \\
}

Here we assume that $T$ is dependent and that forking equals dividing
over $B$. We shall prove that $\ind^{f}$ satisfy all the demands
that appear in (1) in Theorem \ref{thm:MainAbs}. Note that by \ref{cla:indI=00003DindF},
$\ind^{f}=\ind^{i}$, and $\ind^{f}$ is standard. We are left with
showing that $B$ is an extension base for $\ind^{f}$ and that there
is left extension over $B$.\\
Since no type divides over its domain, we get
\begin{claim}
\label{cla:BasisNF}(No need for NIP) $B$ is an extension base for
$\ind^{f}$.
\end{claim}

\begin{claim}
\label{cla:LeftExt}(No need for NIP) We have left extension for $\ind^{f}$
over $B$.\end{claim}
\begin{proof}
Suppose $a\ind_{B}^{f}b$ and we have some $c$. We want to find some
$c'\equiv_{Ba}c$ such that $c'a\ind_{B}^{f}b$. Let $p=\tp\left(c/Ba\right)$.
We need to show that the following set is consistent: \[
p\left(x\right)\cup\left\{ \neg\varphi\left(x,a,b\right)\left|\varphi\mbox{ is over }B\mbox{ and }\varphi\left(x,y,b\right)\mbox{ divides over }B\right.\right\} .\]
 If not, then $p\left(x\right)\vdash\bigvee_{i<n}\varphi_{i}\left(x,a,b\right)$
where $\varphi_{i}\left(x,y,b\right)$ divides over $B$.\\
 So $\psi\left(x,y,b\right):=\bigvee_{i<n}\varphi_{i}\left(x,y,b\right)$
forks over $B$, hence divides over $B$. Assume that $I=\left\langle b_{i}\left|i<\omega\right.\right\rangle $
is an indiscernible sequence that witnesses dividing (with $b_{0}=b$).
By \ref{fac:dividing}, there is $I'\equiv_{Bb}I$ such that $I'$
is indiscernible over $Ba$ and wlog $I'=I$. The type $p$ is over
$Ba$, so $p\left(x\right)\vdash\psi\left(x,a,b_{i}\right)$ for all
$i$. But this is a contradiction as $p$ is consistent.

This concludes the proof of \ref{thm:MainAbs}.
\end{proof}

\subsection*{More conclusion from forking = dividing.\protect \\
}

Here there are no assumption on the theory $T$.
\begin{lem}
Assume forking equals dividing over $B$. Then we have
\begin{enumerate}
\item $a\ind_{B}^{f}a$ iff $a\in\acl\left(B\right)$.
\item $a\ind_{B}^{f}b$ iff $a\ind_{\acl\left(B\right)}^{f}b$ iff $\acl\left(Ba\right)\ind_{B}^{f}b$
iff $a\ind_{B}^{f}\acl\left(Bb\right)$.
\end{enumerate}
\end{lem}
\begin{proof}
(2): Every indiscernible sequence $I$ over $B$ is indiscernible
over $\acl\left(B\right)$: Every 2 increasing sub-sequences from
$I$ have the same Lascar strong type over $B$. As every model containing
$B$ contains $\acl\left(B\right)$, they have the same type over
$\acl\left(B\right)$. It follows that a formula divides over $B$
iff it divides over $\acl\left(B\right)$. Hence $a\ind_{\acl\left(B\right)}^{f}b$
implies $a\ind_{B}^{f}b$. \\
Assume that $a\ind_{B}^{f}b$, and assume that $I$ is a $B$-indiscernible
sequence starting with $b$. Then there is an indiscernible sequence
$I'\equiv_{Bb}I$ such that $I'$ is indiscernible over $Ba$. So
it is also indiscernible over $\acl\left(Ba\right)$. This shows that
$\acl\left(Ba\right)\ind_{B}^{f}b$ (by \ref{fac:dividing}). By right
extension, there is $a'\equiv_{Bb}a$ such that $a'\ind_{B}^{f}\acl\left(Bb\right)$.
But every automorphism fixing $Bb$ pointwise fixes $\acl\left(Bb\right)$
setwise, so $a\ind_{B}^{f}\acl\left(Bb\right)$. By base monotonicity,
we get $a\ind_{\acl\left(B\right)}^{f}b$.\\
The rest follows from monotonicity.\\
(1): Assume that $a\in\acl\left(B\right)$, then since $a\ind_{B}^{f}B$,
it follows from (2) that $a\ind_{B}a$. On the other hand, if $a\ind_{B}^{f}a$,
then the formula $x=a$ does not divide over $B$, so there are not
infinitely many realizations of $\tp\left(a/B\right)$, so this type
is algebraic and we are done.
\end{proof}

\subsection{Applying the previous sections.\protect \\
}

Here we assume $T$ is $\NTPT$ unless stated otherwise.
\begin{cor}
\label{cor:main} Forking equals dividing over models.\end{cor}
\begin{proof}
We use Theorem \ref{thm:MainAbs} with $\ind=\ind^{u}$. We saw in
\ref{cla:CoHeir} that $\ind^{u}$ satisfies all the demands.
\end{proof}
We saw that if the conditions of Theorem \ref{thm:MainAbs} on the
existence of $\ind$ and $B$ are met, then forking equals dividing,
and moreover $B$ is an extension base for $\ind^{\ist}$. So in this
case we can use our version of {}``Kim's lemma''. It gives more
information than just {}``forking equals dividing'', so naturally
we are interested in knowing when this happens. 
\begin{lem}
Suppose $\ind$ is a standard pre-independence relation. Moreover,
assume that every set containing $B$ is an extension base for $\ind$.
Then $\ind$ has left extension over $B$.\end{lem}
\begin{proof}
Assume $a\ind_{B}b$ and we are given $c$. We want to find $c'\equiv_{Ba}b$
such that $ac'\ind_{B}b$. Well, by assumption $c\ind_{Ba}Ba$, so
by right extension there is $c'\equiv_{Ba}c$ such that $c'\ind_{Ba}Bab$.
This means that $c'\ind_{Ba}b$, so by transitivity we get $c'a\ind_{B}b$
as requested. \end{proof}
\begin{defn}
\label{def:goodbase}If $B$ satisfies the condition of the previous
lemma, we say that $B$ is a \emph{good extension base}.\end{defn}
\begin{cor}
\label{cor:GoodBasis}If $B$ is a good extension base for a standard
pre-independence relation $\ind$, and in addition $\ind$ is at least
as strong as $\ind^{i}$, then $B$ is a good extension base for $\ind^{\ist}$
as well. In particular, forking equals dividing over $B$. 
\end{cor}
For instance, this corollary is true if $B$ is a good extension base
for $\ind^{i}$. In dependent theories, since $\ind^{i}=\ind^{f}$,
we have
\begin{cor}
If $T$ is dependent and for every $A$ and $p\in S\left(A\right)$,
$p$ does not fork over $A$, then every set is an extension base
for $\ind^{\ist}$ and forking equals dividing.
\end{cor}
This corollary is true for $o$-minimal theories and $c$-minimal
theories (see \cite[2.14]{HP}).\\

Now we turn to the proof of the main Theorem \ref{thm:MainThmGeneral}.
We abandon for a moment our desire to find extension basis for $\ind^{\ist}$
and concentrate on forking and dividing. In the end we shall conclude
a corollary which is stronger than both \ref{cor:main} and \ref{cor:GoodBasis}.
\begin{claim}
\label{cla:ForkingPreserves} ($T$ any theory) Assume that $a\ind_{B}^{f}b$
and $\varphi\left(x,b\right)$ forks over $B$, then $\varphi\left(x,b\right)$
forks over $Ba$ as well.\end{claim}
\begin{proof}
Assume $\varphi\left(x,b\right)$ forks over $B$, so there are $n<\omega$,
$\varphi_{i}\left(x,y_{i}\right)$ and $b_{i}$ for $i<n$ such that
$\varphi_{i}\left(x,b_{i}\right)$ divides over $B$ and $\varphi\left(x,b\right)\vdash\bigvee_{i<n}\varphi_{i}\left(x,b_{i}\right)$.
By extension, we may assume $a\ind_{A}^{f}b\left\langle b_{i}\left|i<n\right.\right\rangle $.
By \ref{fac:dividing}, $\varphi_{i}\left(x,b_{i}\right)$ divides
over $Ba$. Hence $\varphi\left(x,b\right)$ forks over $Ba$.\end{proof}
\begin{thm}
\label{thm:MainForking} For a set $B$ the following are equivalent:
\begin{enumerate}
\item Forking equals dividing over $B$.
\item $B$ is an extension base for $\ind^{f}$ (i.e. types over $B$ do
not fork over $B$).
\item $\ind^{f}$ has left extension over $B$.
\end{enumerate}
\end{thm}
\begin{proof}
We saw that (1) implies (2) and (3) in \ref{cla:BasisNF} and \ref{cla:LeftExt}.
Assume that (2) or (3) are true. Assume that $\varphi\left(x,a\right)$
forks over $B$, and let $M$ be any model containing $B$.\\
If (2) is true then $M\ind_{B}^{f}B$, so by right extension we
may assume wlog that $M\ind_{B}^{f}a$.\\
If (3) is true, then $B\ind_{B}^{f}a$ (even $B\ind_{B}^{u}a$).
So by left extension we can assume wlog that $M\ind_{B}^{f}a$. \\
So in both cases we are in a situation where we have a model $M$
that satisfies $M\ind_{B}^{f}a$. Hence, by \ref{cla:ForkingPreserves},
$\varphi\left(x,a\right)$ forks over $M$. By \ref{cor:main}, $\varphi\left(x,a\right)$
divides over $M$, so it also divides over $B$.
\end{proof}
The next corollary is stronger than both \ref{cor:main} and \ref{cor:GoodBasis}:
\begin{cor}
\label{cor:inviffist}A set $B$ is an extension base for $\ind^{\ist}$
iff it is an extension base for $\ind^{i}$. In this case, by the
previous theorem, forking equals dividing over $B$.\end{cor}
\begin{proof}
If $B$ is an extension base for $\ind^{\ist}$, it is an extension
base for $\ind^{i}$ by definition. On the other hand, if $B$ is
an extension base for $\ind^{i}$, then, since $\ind^{i}$ is at least
as strong as $\ind^{f}$, $B$ is an extension base for $\ind^{f}$,
so forking equals dividing over $B$ by the previous theorem. By corollary
\ref{cor:istExist}, we are done (since if $\varphi\left(x,a\right)$
forks over $B$, it divides over $B$ so it quasi-divides over $B$). 
\end{proof}

\subsection{Some corollaries for dependent theories.\protect \\
}

Assume $T$ is dependent. We shall see some consequences about the
behavior of forking.
\begin{thm}
The following are equivalent for $B$:
\begin{enumerate}
\item Forking equals dividing over $B$.
\item $B$ is an extension base for $\ind^{f}$.
\item $\ind^{f}$ has left extension over $B$.
\item $B$ is an $\ind^{\ist}$ extension base.
\end{enumerate}
\end{thm}
\begin{proof}
(1) -- (3) are equivalent by \ref{thm:MainForking}. If $B$ is an
extension base for $\ind^{\ist}$, then it is an extension base for
$\ind^{f}$, and we are done by the same theorem. Recall that in a
dependent theory $\ind^{f}=\ind^{i}$, so if $B$ is an extension
base for $\ind^{f}$, it is an extension base for $\ind^{i}$, so
by \ref{cor:inviffist}, also for $\ind^{\ist}$.
\end{proof}
Assume from now on that forking equals dividing over $B$ (for instance,
$B$ is a model).
\begin{cor}
The following are equivalent for a formula $\varphi\left(x,a\right)$:
\begin{itemize}
\item $\varphi$ forks over $B$.
\item $\varphi$ quasi Lascar divides over $B$: there are $\left\{ e_{i}\left|i<m\right.\right\} $
such that $e_{i}\equiv_{B}^{L}a$ and $\left\{ \varphi\left(x,e_{i}\right)\right\} $
is inconsistent.
\end{itemize}
\end{cor}
\begin{proof}
If $\varphi\left(x,a\right)$ forks over $B$, then it quasi Lascar
divides because forking equals dividing over $B$. If $\varphi\left(x,a\right)$
does not fork over $B$, then extend it to $p$, a global non forking
type over $B$. By dependence, $p$ is  Lascar invariant over $B$.
This means that it contains all Lascar conjugates of $\varphi$ over
$B$, and in particular it is impossible for $\varphi$ to quasi Lascar
divide.\end{proof}
\begin{defn}
We say that dividing over $B$ is type definable when for every formula
$\varphi\left(x,y\right)$ there is a (partial) type $\pi\left(x\right)$
over $B$ such that $\pi\left(a\right)$ iff $\varphi\left(x,a\right)$
divides over $B$.\end{defn}
\begin{rem}
Dividing is type definable, so in dependent theories all these notions
-- dividing, forking and quasi Lascar dividing -- are type-definable
over $B$ (i.e. dependent theories are low, see \cite{Buechler})\end{rem}
\begin{proof}
(Due to Itai Ben Yaacov) First we shall see that for any set $B$,
if $\varphi\left(x,a\right)$ divides over $B$ then it $k$ divides
over $B$, with $k=\alt\left(\varphi\right)$. If $\left\langle a_{i}\left|i<\omega\right.\right\rangle $
is an indiscernible sequence witnessing $m>k$ dividing but not $k$
dividing, it means that $\exists x\bigwedge_{i<k}\varphi\left(x,a_{i}\right)$,
and by indiscernibility, $\exists x\bigwedge_{i<k}\varphi\left(x,a_{mi}\right)$.
So assume $\varphi\left(c,a_{mi}\right)$ for $i<k$. But for each
$i$, there must be some $mi<j_{i}\leq mi+m-1$ such that $\neg\varphi\left(c,a_{j_{i}}\right)$.
This is a contradiction to the definition of the alternation rank
(see definition \ref{def:altRank}).\\
The remark now follows: The type $\pi\left(y\right)$ says that
there exists a sequence $\left\langle y_{i}\left|i<\omega\right.\right\rangle $
of elements having the same type as $y$ over $B$, and that every
subset of size $k$ of formulas of the form $\varphi\left(x,y_{i}\right)$
is inconsistent. 
\end{proof}
The following is a strengthening of \cite[Lemma 8.10]{HP}
\begin{cor}
Let $r$ be a partial type which is Lascar invariant over $B$. Then
there exists some global $B$-Lascar invariant extension of $r$.\end{cor}
\begin{proof}
If $\varphi_{1},\ldots,\varphi_{n}\in r$, then $\bigwedge_{i}\varphi_{i}$
does not quasi Lascar divide over $A$ (because all the conjugates
of $\varphi_{i}$ are in $r$ for all $i$). Hence $r$ does not fork
over $B$, hence there is a global non-forking (hence Lascar invariant)
extension.
\end{proof}

\section{\label{sec:BddForking+NTP2}bounded non-forking + $\NTPT$ = Dependent}

It is well-known that stable theories can be characterized as those
simple theories in which every type over model has boundedly many
non-forking extensions (see e.g. \cite[theorem 45]{Ad}). Our aim
in this section is to prove a generalization of this fact: if non-forking
is bounded, and the theory is $\NTPT$, then the theory is actually
dependent. This gives a partial answer to a question of Adler.
\begin{defn}
We say that a pre-independence relation $\ind$ is bounded if there
is a function $f$ on cardinals such that for every type $p\left(x\right)\in S\left(C\right)$
(where $x$ is a finite tuple), and every model $M\supseteq C$, the
size of the set\[
\left\{ \tp\left(a/M\right)\left|a\models p\,\&\, a\ind_{C}M\right.\right\} \]
 is bounded by $f\left(\left|T\right|+\left|C\right|\right)$.
\end{defn}
We quote from \cite[Corollary 38]{Ad}:
\begin{fact}
The following are equivalent for a theory $T$:\end{fact}
\begin{enumerate}
\item $\ind^{f}$ is bounded.
\item $\ind^{f}$ is bounded by the function $f\left(\kappa\right)=2^{2^{\kappa}}$.
\item $\ind^{f}=\ind^{i}$.
\end{enumerate}
The question Adler asks in \cite{Ad} is whether it is true that $T$
is dependent iff $\ind^{f}$ is bounded. The answer in general is
no (see \cite{ArtyomTemp}), but under the assumption of $\NTPT$
it is true.
\begin{thm}
\label{thm:bddforkingThm}Assume $T$ is $\NTPT$, and that $\ind^{f}$
is bounded. Then $T$ is dependent.\end{thm}
\begin{proof}
Assume $\varphi\left(x,y\right)$ has the independence property. This
means that there is an infinite set $A$ of tuples, such that for
any subset $B\subseteq A$, there is some $b$ such that for all $a\in A$,
$\varphi\left(b,a\right)$ iff $a\in B$. Let $r\left(x\right)=\left\{ x\neq a\left|a\in A\right.\right\} $
be a partial type over $A$. Since it is finitely satisfiable in $A$
there is a global type $p$ containing $r$ which is finitely satisfied
in $A$. Let $q=p^{\left(2\right)}$. Denote $\psi\left(x,y,z\right)=\varphi\left(x,y\right)\land\neg\varphi\left(x,z\right)$.\\
Note that if $M\supseteq A$ is a model and $b\equiv_{M}c$ then
$\psi\left(x,b,c\right)$ forks over $M$ (otherwise there is a global
non-forking type over $M$ which is not invariant over $M$ in contradiction
to our assumption) and hence divides over $M$. \\
We build by induction on $\alpha<\omega_{1}$ a sequence of indiscernible
sequences $J_{\alpha}=\left\langle I_{i}\left|i<\alpha\right.\right\rangle $
such that
\begin{enumerate}
\item $J_{\alpha'}\subseteq J_{\alpha}$ for $\alpha'<\alpha$.
\item $I_{i}=\left\langle a_{i,j}\left|j<\omega\right.\right\rangle $.
\item For all $i<\alpha$, $j<\omega$, $a_{i,j}\models q|_{AJ_{i}}$.
\item For all $i<\alpha$, $I_{i}$ witnesses the dividing of $\psi\left(x,a_{i,0}\right)$
(over $\emptyset$). 
\end{enumerate}
For $\alpha=0$ there is nothing to do, for $\alpha$ limit we take
the union.\\
For $\alpha+1$: Let $M$ be a model containing $AJ_{\alpha}$.
Let $a_{\alpha,0}\models q|_{M}$. Then $\psi\left(x,a_{\alpha,0}\right)$
divides over $M$, and let $I_{\alpha}$ witness this. It is easy
to see that all demands are met.\\
Since the array is of length $\omega_{1}$, there is some $k$
such that for infinitely many $i<\omega_{1}$, $I_{i}$ witnesses
$k$-dividing . Wlog, these are the first $\omega$. It follows that
for every vertical path $\eta:\omega\to\omega$, $\tp\left(\left\langle a_{i,\eta\left(i\right)}\left|i<\omega\right.\right\rangle /A\right)=q^{\left(\omega\right)}|_{A}$.\\
Now we shall show that the set $\left\{ \psi\left(x,a_{i,0}\right)\left|i<\omega\right.\right\} $
is consistent and reach a contradiction to $\NTPT$.\\
Denote $a_{i}=a_{i,0}=\left(b_{i},c_{i}\right)$. Note that by
the choice of $p$ and $q$, for every formula $\phi\left(x_{0},y_{0},\ldots,x_{n-1},y_{n-1}\right)$,
if $\phi\left(a_{0},\ldots,a_{n-1}\right)$, then there are pairwise
distinct $b_{0}',c_{0}',\ldots,b_{n-1}',c_{n-1}'\in A$ such that
\[
\phi\left(b_{0}',c_{0}',\ldots,b_{n-1}',c_{n-1}'\right).\]
For $n<\omega$, let $\phi=\neg\exists x\bigwedge_{i<n}\psi\left(x,a_{i}\right)$,
then there are pairwise distinct \\
$b_{0}',c_{0}',\ldots,b_{n-1}',c_{n-1}'\in A$ such that $\neg\exists x\bigwedge_{i<n}\psi\left(x,b_{i}',c_{i}'\right)$,
which contradicts the choice of $\psi$, i.e. this set is consistent. 
\end{proof}

\section{\label{sec:optimality-of-results}optimality of results}

In general, forking is not the same as dividing, and Shelah already
gave an example in \cite[III,2]{Sh:c}. Kim gave another example in
his thesis (\cite[Example 2.11]{KimThesis}) -- circular ordering.
Both examples were over the empty set, and the theory was dependent.\\
Here we give $2$ examples. The first shows that outside the realm
of $\NTPT$, our results are not necessarily true, and the second
shows that even in dependent theories, forking is not the same as
dividing even over sets containing models.\\
In both examples, we use the notion of a (directed) circular order,
so here is the definition:
\begin{defn}
A circular order on a finite set is a ternary relation obtained by
placing the points on a circle and taking all triples in clockwise
order. For an infinite set, a circular order is a ternary relation
such that the restriction to any finite set is a circular order. \\
A first order definition is: a circular order is a ternary relation
$C$ such that for every $x$, $C\left(x,-,-\right)$ is a linear
order on $\left\{ y\left|y\neq x\right.\right\} $ and $C\left(x,y,z\right)\to C\left(y,z,x\right)$
for all $x,y,z$. 
\end{defn}

\subsection{Example 1.}

Here we present a variant of an example found by Martin Ziegler, showing
that
\begin{enumerate}
\item forking and dividing over models are different in general,
\item strictly non-forking types need not exist over models (see \ref{rem:StrictNF}),
so in particular, strictly invariant types and non-forking heirs need
not necessarily exist over models.
\end{enumerate}
Let $L$ be a $2$ sorted language: one sort $P$ for \textquotedbl{}points\textquotedbl{},
for which we will use the variables $t,t_{0},\ldots$ and another
$S$ for \textquotedbl{}sets\textquotedbl{}, for which we will use
the variables $s,s_{0},\ldots$. $L$ consists of $1$ binary relation
$E\left(t,s\right)$ to denote \textquotedbl{}membership\textquotedbl{}
(so a subset of $P\times S$), and two 4-ary relations: $C\left(t_{1},t_{2},t_{3},s\right)$
and $D\left(s_{1},s_{2},s_{3},t\right)$.\\
Consider the following universal theory $T^{\forall}$ saying:
\begin{enumerate}
\item For all $s$, $C\left(-,-,-,s\right)$ is a circular order on the
set of all $t$ such that $E\left(t,s\right)$, and if $C\left(t_{1},t_{2},t_{3},s\right)$
then $E\left(t_{i},s\right)$ for $i=1,2,3$, and
\item For all $t$, $D\left(-,-,-,t\right)$ is a circular order on the
set of all $s$ such that $\neg E\left(t,s\right)$, and if $D\left(s_{1},s_{2},s_{3},t\right)$
then $\neg\left(E\left(t,s_{i}\right)\right)$ for $i=1,2,3$. 
\end{enumerate}
This theory has the joint embedding property and the amalgamation
property as can easily be verified by the reader. Hence, as the language
has no function symbols, by Fraïssé's theorem it has a model completion
$T$, so $T$ eliminates quantifiers (see \cite[Theorem 7.4.1]{Hod}).\\
Let $M$ be a model of $T$. We choose $t_{0},s_{0}\in\C\backslash M$,
such that for all $t\in M$, $\neg E\left(t,s_{0}\right)$ and for
all $s\in M$, $E\left(t_{0},s\right)$. Now, $E\left(x,s_{0}\right)$
forks over $M$, and $\neg E\left(t_{0},y\right)$ forks over $M$,
but none of them (quasi) divides.\\
Why? Non quasi dividing is straightforward from the construction
of $T$.\\
We show that $\neg E\left(t_{0},y\right)$ forks (for $E\left(x,s_{0}\right)$
use the same argument): choose some circular order on $P^{M}$, and
choose $s_{i}'$ for $i<\omega$ such that: 
\begin{itemize}
\item $\neg E\left(t_{0},s'_{i}\right)$ for $i<\omega$.
\item $D\left(s'_{i},s'_{j},s'_{k},t_{0}\right)$ whenever $i<j<k$.
\item For all $i<\omega$ and for all $t\in M$ we have $E\left(t,s'_{i}\right)$,
and $C\left(-,-,-,s'_{i}\right)$ orders $P^{M}$ using the pre-chosen
circular order.
\end{itemize}
Now, \[
\neg E\left(t_{0},y\right)\vdash D\left(s_{0}',y,s'_{1},t_{0}\right)\vee D\left(s_{1}',y,s_{0}',t_{0}\right)\vee y=s_{0}'\vee y=s_{1}'\]
 and $D\left(s_{0}',y,s_{1}',t_{0}\right)$ divides over $Mt_{0}$
as witnessed by $\left\langle s_{i}'s_{i+1}'\left|i<\omega\right.\right\rangle $,
and so does $D\left(s_{1}',y,s_{0}',t_{0}\right)$, because for all
$n$, $s_{1}'s_{0}'\equiv_{Mt_{0}}s_{n+1}'s_{n}'$. \\
Let $p\left(t\right)$ be $\tp\left(t_{0}/M\right)$. We show that
$p$ is not a strictly non-forking type over $M$: suppose $q$ is
a global strictly non-forking extension, and let $t_{0}'\models q|_{s_{0}}$.
Then $t_{0}'\ind_{M}^{f}s_{0}$ and $s_{0}\ind_{M}^{f}t_{0}'$. So
surely $\neg E\left(t,s_{0}\right)\in q$, so $\neg E\left(t_{0}',s_{0}\right)$
holds. But $t_{0}'\equiv_{M}t_{0}$ so $s_{0}\nind_{M}^{f}t_{0}'$
-- a contradiction.\\
Note that $T$ has the tree property of the second kind: Let $s_{i}$
for $i<\omega$ be such that they are all different, and for each
$i$, let $t_{j}^{i}$ for $j<\omega$, be such that for $j<k<l$,
$C\left(t_{j}^{i},t_{k}^{i},t_{l}^{i},s_{i}\right)$. The array $\left\{ C\left(t_{j}^{i},x,t_{j+1}^{i},s_{i}\right)\left|i,j<\omega\right.\right\} $
witnesses $\mbox{TP}_{\mbox{2}}$.

\subsection{Example 2.}

We give an example showing that even if $T$ is dependent, and $S$
contains a model, forking is not necessarily the same as dividing
over $S$. Hence models are not good extension bases for non-forking
in dependent theories in general (see \ref{def:goodbase}).\\
Let $L$ the language $\left\{ C,E\right\} $ where $E$ is a binary
relation and $C$ is a ternary relation. Let $T^{\forall}$ be the
universal theory saying that $E$ is an equivalence relation and that
$C$ induces a circular order on every equivalence class, and that
in addition $\forall x,y,z\left(C\left(x,y,z\right)\to E\left(x,y\right)\land E\left(y,z\right)\right)$.\\
This theory has the JEP and AP so it has a model completion (as
in Example 1). \\
Moreover, $T$ is dependent: To show this, it's enough to show
that all formulas $\varphi\left(x,y\right)$ where $x$ is one variable
have finite alternation rank. As $T$ eliminates quantifiers, it's
enough to consider atomic formulas (see e.g. \cite[Section 1]{Ad}),
and this is straightforward and left to the reader.\\
Consider $T^{\eq}$. It is also dependent. \\
Let $M$ be a model. Let $c\in\C\backslash M$ be a code of an
$E$-equivalence class without any $M$-points. Then for every $a_{1}\neq a_{2}$
in this class, both $C\left(a_{2},x,a_{1}\right)$ and $C\left(a_{1},x,a_{2}\right)$
divide over $Mc$ (like in Example 1). So we have \[
\pi_{E}\left(x\right)=c\vdash C\left(a_{1},x,a_{2}\right)\vee C\left(a_{2},x,a_{1}\right)\vee x=a_{1}\vee x=a_{0}\]
 forks but does not divide over $Mc$ (where $\pi_{E}$ is the canonical
projection into the sort of codes of $E$-classes).

\section{Further remarks}

Our understanding of forking in dependent theories was highly influenced
by Section 5 (Non-forking) in \cite{Sh783}. This section contains
the definition of strict non-forking, that we generalized to $\ind^{\ist}$
(in dependent theories they are equal). Essentially, the ideas of
the proof of Lemma \ref{lem:Kim} ({}``Kim's Lemma'') appears there.
Alex Usvyatsov also noticed a variant of that lemma independently.\\
The claim and proof of \ref{cla:globalTypeWitness}, with some
modifications and generalizations is due to Usvyatsov and Onshuus
in \cite{OnUs1}. It should be noted that H. Adler and A. Pillay were
the first to realize that $\NTPT$ is all the assumption one needs
there.\\
Alex Usvyatsov noticed that one can use the broom lemma to prove
that types over models can be extended to global non-forking heirs
(see \cite{Us1}). In fact, this follows directly from \ref{cla:symmetrizers}.

\section{\label{sec:Questions-and-remarks}Questions and remarks}
\begin{enumerate}
\item Are simple theories $\ind^{i}$-extensible $\NTPT$ theories?
\item Can similar results be proved for $\mbox{NSOP}$ theories? Or at least
$\mbox{NTP}_{1}$ theories? 
\item It would be nice to find some purely semantic characterization of
theories in which forking equals dividing over models. For example
we know that all $\NTPT$ theories are such, however the opposite
is not true: there is a theory with $\mbox{TP}_{\mbox{2}}$ in which
forking equals dividing (essentially the example from section \ref{sec:optimality-of-results},
but with dense linear orders instead of circular ones).
\end{enumerate}
\bibliographystyle{alpha}
\bibliography{common}

\address{Artem Chernikov\\
Université Lyon 1\\
CNRS, Institut Camille Jordan UMR5208\\
43 boulevard du 11 novembre 1918\\
69622 Villeurbanne Cedex\\
France\\
}

\email{\texttt{artyom.chernikov@gmail.com }~\\
}

\address{Itay Kaplan\\
Fachbereich Mathematik und Statistik\\
Universität Konstanz \\
78457 Konstanz\\
Germany}

\email{\texttt{itay.kaplan@uni-konstanz.de}}
\end{document}